%% file: subcritical009.tex
\documentclass[11pt]{article}
\usepackage{amsfonts,amsmath,amsthm,amssymb}
\usepackage{relsize}
\usepackage{color}
\usepackage{appendix}
\usepackage{ulem}
\usepackage[makeroom]{cancel}
\usepackage[margin=1.4in]{geometry}
\usepackage{stmaryrd}
\newenvironment{ack}{\medskip{\it Acknowledgement.}}{}

\allowdisplaybreaks

\let\TeXchi\chi
\newbox\chibox
\setbox0 \hbox{\mathsurround0pt $\TeXchi$}
\setbox\chibox \hbox{\raise\dp0 \box 0 }
\def\chi{\copy\chibox}

\input dibe_1009.tex
\input harnack_mono.tex

\begin{document}
%%%%%%%%%%%%%%%%%%%%%%%%%%%%%%%%%%%%%%%%%%%%%%%%%%%
\title{A Boundary Estimate for Singular Sub-Critical\\ Parabolic Equations}
%%%%%%%%%%%%%%%%%%%%%%%%%%%%%%%%%%%%%%%%%%%%%%%%%%%
%%%%%%%%%%%%%%%%%%%%%%%%%%%%%%%%%%%%%%%%%%%%%%%%
\author{Ugo Gianazza\\
Dipartimento di Matematica ``F. Casorati", 
Universit\`a di Pavia\\ 
via Ferrata 5, 27100 Pavia, Italy\\
email: {\tt gianazza@imati.cnr.it}
%%%%%%%%%%%%%%%%%%%%%%%%%%%%%%%%%%%%%%%%%%%%%%%%%%
\and Naian Liao\\
%\thanks{Corresponding author; N. L. is partially supported by NSFC 11701054 grant.}\\
%%%%%%%%%%%%%%%%%%%%%%%%%%%%%%%%%%%%
Fachbereich Mathematik, Universit\"at Salzburg\\
Hellbrunner Str. 34, 5020 Salzburg, Austria\\
email: {\tt naian.liao@sbg.ac.at}
}
%%%%%%%%%%%%%%%%%%%%%%%%%%%%%%%%%%%%%%%%%%%%%%%%%%%
\date{}
\maketitle
%%%%%%%%%%%%%%%%%%%%%%%%%%%%%%%%%%%%%%%%%%%%%%%%%%%
\begin{abstract}
We prove an estimate on the modulus of continuity at a boundary point of a cylindrical domain for local weak solutions to singular parabolic equations of $p$-laplacian type, with $p$ in the sub-critical range $(1,\frac{2N}{N+1}]$. The estimate is given in terms of a Wiener-type integral, defined by a proper elliptic $p$-capacity. 
%%%%%%%%%%%%%%%%%%%%%%%%%%%%%%%%%%%%%%%%%%%%%%%%%%%
\vskip.2truecm
%%%%%%%%%%%%%%%%%%%%%%%%%%%%%%%%%%%%%%%%%%%%%%%%%%%
\vskip.2truecm
%%%%%%%%%%%%%%%%%%%%%%%%%%%%%%%%%%%%%%%%%%%%%%%%%%%
\end{abstract}
%%%%%%%%%%%%%%%%%%%%%%%%%%%%%%%%%%%%%%%%%%%%%%%%%%%%%%%%%%%%%%%%%%%%%%%
%%%%%%%%%%%%%%%%%%%%%%%%%%%%%%%%%%%%%%%%%%%%%%%%%%%%%%%%%%%%%%%%%%%%%%%
\section{Introduction}%\label{S:intro}
%%%%%%%%%%%%%%%%%%%%%%%%%%%%%%%%%%%%%%%%%%%%%%%%%%%%%%%%%%%%%%%%%%%%%
Let $E$ be an open set in $\rn$, and for $T>0$ let $E_T$ denote the
cylindrical domain $E\times[0,T]$. Moreover let
\begin{equation*}
S_T=\partial E\times(0,T],\qquad \partial_p E_T=S_T\cup(\bar{E}\times\{0\})
\end{equation*}
denote the lateral, and the parabolic boundary respectively.

We shall consider quasi-linear, parabolic partial differential equations of the form
\begin{equation}  \label{Eq:1:1}
u_t-\dvg\bl{A}(x,t,u, Du) = 0\quad \text{ weakly in }\> E_T,
\end{equation}
%We make the following hypotheses: 
%The function $\bl{A}:E_T\times\rr^{N+1}\to\rn$ is only assumed to be
%measurable and subject to the structure conditions
where the function $\bl{A}(x,t,u,\xi)\colon E_T\times\rr^{N+1}\to\rn$ is assumed to be
measurable with respect to $(x, t) \in E_T$ for all $(u,\xi)\in \rr\times\rn$,
and continuous with respect to $(u,\xi)$ for a.e.~$(x,t)\in E_T$.
Moreover, we assume the structure conditions
\begin{equation}  \label{Eq:1:2}
\left\{
\begin{array}{l}
\bl{A}(x,t,u,\xi)\cdot \xi\ge C_o|\xi|^p \\
|\bl{A}(x,t,u,\xi)|\le C_1|\xi|^{p-1}%
\end{array}%
\right .\quad \text{ a.e.}\> (x,t)\in E_T,\, \forall\,u\in\rr,\,\forall\xi\in\rn
\end{equation}
where $C_o$ and $C_1$ are given positive constants, and we take
\begin{equation}\label{Eq:p-range}
1<p\le\frac{2N}{N+1}\df=p_*.
\end{equation}
The number $p_*$ is referred to as a {\it critical} value
in the local regularity theory for equations \eqref{Eq:1:1}--\eqref{Eq:1:2} (see \cite{DBGV-mono}). 
The range \eqref{Eq:p-range} is often called {\it  singular, sub-critical},
whereas $p_*<p<2$ is {\it singular, super-critical}.

In addition, the 
principal part $\bl{A}$ is assumed to be monotone in the variable $\xi$ 
in the sense
\begin{equation}\label{Eq:5:1:3}
(\bl{A}(x,t,u,\xi_1)-\bl{A}(x,t,u,\xi_2))
\cdot(\xi_1-\xi_2)\ge0
\end{equation}
for all variables in the indicated domains, and Lipschitz continuous 
in the variable $u$, that is,
\begin{equation}\label{Eq:5:1:4}
|\bl{A}(x,t,u_1,\xi)-\bl{A}(x,t,u_2,\xi)|\le \Lm|u_1-u_2| (1+|\xi|^{p-1})
\end{equation}
for some given $\Lm>0$, and for the variables in the indicated domains.

Important examples of partial differential equations satisfying \eqref{Eq:1:1}--\eqref{Eq:5:1:4} include
the parabolic $p$-laplacian
\[
u_t-\dvg\big(|Du|^{p-2}Du\big)=0
\]
and more generally
\[
u_t-\sum_{i,j=1}^N\big(a_{ij}(x,t)|Du|^{p-2}u_{x_i}\big)_{x_j}=0,
\]
where $(a_{ij})$ is a positive definite $N\times N$ matrix with bounded and measurable entries.

Let us consider a boundary datum 
\begin{equation}\label{Eq:bound}
\left\{
\begin{aligned}
&g\in L^p\big(0,T;W^{1,p}( E)\big),\\ 
&g \text{ continuous on}\ \overline{E}_T\ \text{with modulus of continuity }\ \om_g(\cdot).
\end{aligned}
\right.
\end{equation}
We are interested in the boundary behavior of solutions to the Cauchy-Dirichlet problem
\begin{equation}\label{Eq:1:6}
\left\{
\begin{aligned}
&u_t-\dvg\bl{A}(x,t,u, Du) = 0\quad \text{ weakly in }\> E_T\\
&u(\cdot,t)\Big|_{\partial E}=g(\cdot,t)\quad \text{ a.e. }\ t\in(0,T]\\
&u(\cdot,0)=g(x,0),
\end{aligned}
\right.
\end{equation}
with $g$ as in \eqref{Eq:bound}. We do not impose any {\it a priori} requirements on the boundary of the domain
$E\subset\rn$. 

Notice that when $p$ is in the sub-critical range \eqref{Eq:p-range},
the boundedness of a weak solution to \eqref{Eq:1:6} 
is not guaranteed by the mere notion of weak solution. A detailed discussion is given, for example in \cite[Chapter~V]{dibe-sv}, or in \cite[Chapter~6, Section~21.3]{DBGV-mono}.
Therefore, we directly assume the boundedness of solutions.

We refer to the parameters $\datap$ as our structural data, and we write $\gm%
=\gm(p,N,C_o,C_1)$ if $\gm$ can be quantitatively
determined a priori only in terms of the above quantities.
The generic constant $\gm$ may change from line to line.

For $x_o\in \rn$ and $\rho > 0 $, $K_{\rho}(x_o)$ denotes the cube of edge $2\rho$, centered at $x_o$
with faces parallel to the coordinate planes. When $x_o$ is the
origin of $\rn$, we simply write $K_{\rho}$.

Let $\pto\in S_T$, and for $R_o\in(0,1)$ we set the backward, space-time cylinder
\[
Q_{R_o}=K_{R_o}(x_o)\times(t_o-2R_o^p,t_o],
\]
where $R_o$ is so small that $(t_o-2R_o^p,t_o]\subset(0,T]$.
Moreover, we define
\[
\mu_o^+=\essup_{Q_{R_o}\cap E_T}u,\qquad\mu_o^-=\essinf_{Q_{R_o}\cap E_T}u,
\qquad\om_o=\mu_o^+-\mu_o^-=\essosc_{Q_{R_o}\cap E_T}u. 
\]
We will give the formal definition of solution to \eqref{Eq:1:6} in \S~\ref{S:1:2}.
Now we proceed to state the main result of this paper.
\begin{theorem}\label{Thm:1:1}
Let $u$ be a bounded, weak solution to \eqref{Eq:1:6}, and assume that 
\eqref{Eq:1:2}--\eqref{Eq:bound} hold. Then there exist positive constants $\gm,\,c,\,\al$, and $q_o>\frac1{p-1}$, which depend only on the data $\datap$, such that
\begin{equation}\label{Eq:1:9}
\essosc_{Q_{\rho}(\om_o)\cap E_T}u\le
\om_o\exp\left\{-{\gm}\int_{\rho^\al}^1 \left[\dl(s)\right]^{q_o}\frac{ds}{s}\right\}
+2\osc_{\widetilde{Q}_o(\rho)\cap S_T}g,
\end{equation}
where $0<\rho<R_o$, 
\[
Q_\rho(\om_o)=K_\rho(x_o)\times\left(t_o-\tfrac 12c\om_o^{2-p}\rho^p,t_o\right],
\]
\[
\dl(\rho)=\frac{{\rm cap}_p\big(K_{\rho}(x_o)\backslash E,K_{\frac32\rho}(x_o)\big)}{{\rm cap}_p\big(K_{\rho}(x_o),K_{\frac32\rho}(x_o)\big)},
\]
%$c\in(0,1)$ is the infimum between the quantities defined in \eqref{Eq:lower} and in \eqref{Eq:H:1} below,
$\widetilde{Q}_o(\rho)$ is a proper reference cylinder, which shrinks to $(x_o,t_o)$ as $\rho\to0$, and ${\rm cap}_p(D,B)$ denotes the (elliptic) $p$-capacity of $D$ with respect to $B$.
\end{theorem}
%%%%%%%%%%
\begin{remark}\upshape
In \S~\ref{S:final} we will give a more precise quantification of $q_o$. Its value depends on the Harnack-type inequality we will present in Theorem~\ref{FV}.
\end{remark}
\begin{remark}\upshape
 Although the precise expression of $\widetilde{Q}_o(\rho)$ plays no role, as far as the decay of $u$ at the boundary is concerned,  we present it here for the reader's convenience:
\begin{equation*}%\label{large:cyl}
\widetilde{Q}_o(\rho)\df=
\left\{
\begin{aligned}
&K_{2\tilde r}(x_o)\times[t_o-c\om_o^{2-p}2(2\tilde r)^p, t_o]\\ 
&\text{with}\ \tilde r=[\bar{\om}(\rho^\al)]^{\frac12},
\end{aligned}
\right.
\end{equation*}
where $(0,1)\ni\rho\mapsto\bar{\om}(\rho)$ is defined by
\[
\bar{\om}(\rho)=\exp\left\{-\int_{\rho}^1 \left[\dl(s)\right]^{q_o}\frac{ds}{s}\right\}.
\]
We refer to \cite[\S~6.4]{GLL} for the derivation of $\widetilde{Q}_o(\rho)$.
\end{remark}

A point $(x_o,t_o)\in S_T$ is called a {\it Wiener point} if
$\dsty\int_\tau^1\left[\dl(\rho)\right]^{q_o}\frac{d\rho}\rho\to\infty$
as $\tau\to0$; see Appendix~\ref{S:cap} for more. 
Relying on this definition, {from Theorem~\ref{Thm:1:1} we can conclude
  the following corollary in a standard way.}
\begin{corollary}\label{Cor:1:1}
Let $u$ be a bounded, weak solution to \eqref{Eq:1:6}, assume that \eqref{Eq:1:2}--\eqref{Eq:bound} hold true, and that $(x_o,t_o)\in S_T$ is a Wiener point. Then 
\[
\lim_{\genfrac{}{}{0pt}{}{(x,t)\to(x_o,t_o)}{(x,t)\in E_T}}u(x,t)=g(x_o,t_o).
\]
\end{corollary}

Theorem \ref{Thm:1:1} also implies H\"older regularity up to the
boundary under a fairly weak assumption on the domain, i.e. the complement of $E$
is {\it uniformly $p$-fat}. More
precisely, a set $A$ is {\it uniformly $p$-fat}, if for some
$\gm_o,\,\bar\rho>0$ one has
\[
\frac{{\rm cap}_p\big(K_{\rho}(x_o)\cap A,K_{\frac32\rho}(x_o)\big)}{{\rm
    cap}_p\big(K_{\rho}(x_o),K_{\frac32\rho}(x_o)\big)}\geq \gm_o
\]
for all $0<\rho <\bar\rho$ and all $x_o\in A$. See \cite{lewis1988} for
more on this notion. 
However, we point out next that this conclusion can be achieved merely under the
structural condition \eqref{Eq:1:2} of $\bl{A}$, for all $1<p<2$.

\begin{corollary}\label{cor:holder}
  Let $u$ be a bounded, weak solution to \eqref{Eq:1:6}, assume that 
 \eqref{Eq:1:2} holds true for $1<p<2$, 
that the complement of the domain $E$ is
  uniformly $p$-fat, and let $g$ be H\"older continuous.  Then the
  solution $u$ is H\"older continuous up to the lateral boundary.
\end{corollary}
%%%%%%%%%%%%%%%%%
\subsection{Method of the Proof}
A boundary estimate of the kind \eqref{Eq:1:9} was established for the elliptic $p$-laplacian by Maz'ya \cite{mazya70}. 
The method employed was a potential theoretical one
and heavily depended on the comparison principle. Such an estimate of Maz'ya
was later established for very general elliptic operators with structures similar to \eqref{Eq:1:2}
by Gariepy and Ziemer \cite{Gar-Zie}.  Our approach mainly follows their ideas whose
adaption to the parabolic $p$-laplacian setting presents considerable difficulty, cf.~\cite{GLL, GL}.
For this reason, we recapitulate the main steps in the following.
For simplicity, let us suppose  $u$ is a solution to the elliptic $p$-laplacian in an open set $E\subset \rn$
and it attains zero on part of the boundary $O\subset \pl E$.

The first step consists in showing the truncated function $(u-k)_+$ with $k>0$ is a local sub-solution
in a cube $K_{2\rho}(x_o)$ with the center $x_o\in O$, after a zero extension outside $E$. 
Then we work with the non-negative, local super-solution $v$ in $K_{2\rho}(x_o)$ defined by
\begin{equation}\label{Eq:sketch:0}
 v\df=\mu - (u-k)_+,\quad\text{ where }\mu=\sup_{K_{2\rho}(x_o)}(u-k)_+.
\end{equation}
Here and in the sequel, we use ``$\sup$/$\inf$" instead of ``$\essup$/$\essinf$" for simplicity.

Next, we derive an energy estimate for $v$ of the following type:
\begin{equation}\label{Eq:sketch:1}
\int_{K_{\frac32\rho}(x_o)}|D(v\z)|^p\,dx\le\gm \mu\rho^{N-p}\left(\bint_{K_{\frac32\rho}(x_o)}{v}\,dx\right)^{p-1},
\end{equation}
where $\z\in C^1_o(K_{\frac32\rho}(x_o))$ is a cutoff function that equals $1$ in $K_{\rho}(x_o)$.
Notice that $v\z\in W^{1,p}_o(K_{\frac32\rho}(x_o))$ and $v\z=\mu$ in  $K_{\rho}(x_o)\setminus E$.
Hence by the definition of $p$-capacity in Appendix~\ref{S:cap} we may estimate
\begin{equation}\label{Eq:sketch:2}
\int_{K_{\frac32\rho}(x_o)}|D(v\z)|^p\,dx\ge\mu^p{\rm cap}_p( K_{\rho}(x_o)\setminus E, K_{\frac32\rho}(x_o)).
\end{equation}
Combining \eqref{Eq:sketch:1} and  \eqref{Eq:sketch:2} 
and noting ${\rm cap}_p( K_{\rho}(x_o), K_{\frac32\rho}(x_o))=c\rho^{N-p}$ for some positive $c=c(p,N)$,
 we easily obtain
\begin{equation}\label{Eq:sketch:3}
\mu[\dl(\rho)]^{\frac1{p-1}}\le\gm\bint_{K_{\rho}(x_o)}{v}\,dx.
\end{equation}

Finally, the reduction of oscillation is realized via an application of the weak Harnack inequality
for the non-negative super-solution $v$ in $K_{2\rho}(x_o)$. As a result we arrive at
\begin{equation*}%\label{Eq:sketch:4}
\mu[\dl(\rho)]^{\frac1{p-1}}\le\gm\bint_{K_{\rho}(x_o)}{v}\,dx\le\widetilde{\gm}\inf_{K_{\rho}(x_o)} v,
\end{equation*}
which implies
\[
\sup_{K_{\rho}(x_o)}(u-k)_+\le\mu\left(1-\tfrac1{\widetilde{\gm}} [\dl(\rho)]^{\frac1{p-1}}\right).
\]

Coming back to our parabolic setting, the analog of \eqref{Eq:sketch:0} is defined in \eqref{Eq:super},
which will be proven to be a super-solution across the lateral boundary in Lemma~\ref{Lem:2:0}.
The analog of \eqref{Eq:sketch:3} has been proven in \cite{GLL} and will be recalled in Lemma~\ref{Lm:5:1}.
Due to the lack of a proper weak Harnack inequality in  the singular, sub-critical range of $p$,
we will instead employ a Harnack-type inequality presented in Theorem~\ref{FV}.
This is the stage where we need the comparison principle.
%%%%%%%%%%%%%%%%%
\subsection{Novelty and Significance}
This is the third paper in a wider project devoted to the study of boundary behavior of solutions to \eqref{Eq:1:6}: in \cite{GLL} we dealt with the singular super-critical range $p\in\left(p_*,2\right)$, in \cite{GL} we studied the degenerate case $p>2$, and here we consider the singular sub-critical interval given by $p\in(1,p_*]$, which to our knowledge has never been dealt with before. Some remarks about the significance of our results have already been given in \cite{GLL,GL}, and we will not repeat them here: what we are going to concentrate on, are the differences that the sub-critical range shows with respect to the super-critical one.

In all these three papers, the most interesting result is that a Wiener point is a continuity point for the solution at the boundary, and a quantitative characterization of the decay in a neighborhood of that point at the boundary is given (see \eqref{Eq:1:9} in this case). However, there is a fundamental difference: for $p>p_*$ (see \cite{GLL,GL}), a point $(x_o,t_o)\in S_T$ is a Wiener point if $\dsty\int_0^1[\dl(\rho)]^{\frac1{p-1}}\,\frac{d\rho}{\rho}=\infty$, whereas here we require $\dsty\int_0^1[\dl(\rho)]^{q_o}\,\frac{d\rho}{\rho}=\infty$, with $q_o>\frac1{p-1}$. Hence, our result is not optimal, and our quantitative estimates cannot recover what is already known qualitatively, at least in the prototype case.

Indeed, the fact that a Wiener point with $q_o=\frac1{p-1}$ is a continuity point has already been
observed in \cite[Proposition~5.4]{BBGP} for the prototype parabolic $p$-laplacian
\begin{equation}\label{p-lapl}
u_t-\dvg(|Du|^{p-2}Du)=0, 
\end{equation}
for
{\it any} $p>1$, hence both singular, i.e. with $1<p<2$, and degenerate,
i.e. with $p>2$ (see also \cite{KiLi96}). %\textcolor{red}{
In \cite{BBGP} the characterization is provided in terms of a family of barriers, adapting to the parabolic $p$-laplacian the well-known Perron method. This approach is quite flexible, as it allows us to give some simple geometric conditions which ensure the regularity of boundary points. Here by {\it regular}, we mean a boundary point where all solutions to the Dirichlet problem attain their continuous boundary values continuously. Non-cylindrical domains can then be considered, and indeed Proposition~5.4, mentioned above, deals with quite general sets, of which cylinders as the ones considered here are just a particular instance. On the other hand, Perron's method, at least as developed in \cite{BBGP}, does not give any quantitative modulus of continuity. Moreover, at the moment it is not known, whether the qualitative, geometric characterization of regular points is true for more general operators $\bl A$, like the ones we study here.%} 

The lack of optimality of the exponent $q_o$ notwithstanding, here the novelty with respect to the existing literature is twofold: first of all, we deal with quite general operators, and the only restriction for Theorem~\ref{Thm:1:1} and Corollary~\ref{Cor:1:1} lies in the requirement that the Comparison Principle is satisfied, whereas for Corollary~\ref{cor:holder} the structural condition \eqref{Eq:1:2} suffices; on the other hand, as we have already mentioned above, to our knowledge a quantitative characterization, in the singular sub-critical range, of the boundary behavior of solutions to \eqref{Eq:1:6} for rough sets  has never been provided. 

There is yet another new observation regarding what we had in \cite{GLL}:
 we stated Theorem~1.1 of \cite{GLL} in {\it centered} cylinders.
Here we seize this opportunity to point out that in fact {\it backward} cylinders will also work in \cite{GLL},
just like in Theorem~\ref{Thm:1:1} here. The modification can be modeled on Lemma~\ref{Lm:7:1} below.
This matches the observation made in   
 \cite[Theorem~3.1]{BBGP} for the prototype equation \eqref{p-lapl}, 
 that what happens in the future, namely for $t>t_o$, does not affect
 the regularity of the boundary point $(x_o,t_o)$ for the parabolic $p$-laplacian \eqref{p-lapl}. 
 %It is quite natural to expect an analogous situation, at least for $\bl{A}=\bl{A}(x,Du)$, but we were not able to show it, either in \cite{GLL} or here. 
 
%As a matter of fact, this is not the only similarity between these notes and \cite{GLL}: indeed, 
Although we have basically followed the same approach as in \cite{GLL}, the fundamental difference lies in the Harnack inequalities we have to employ. When $p\in\left(p_*,2\right)$, we can rely upon an intrinsic weak Harnack inequality, and this allows us to conclude; when $p\in\left(1,p_*\right]$, only Harnack-type estimates are at disposal (see Theorem~\ref{FV} below), and this is ultimately the reason for the lack of optimality of $q_o$. It is important to recall that the lack of a proper Harnack inequality when $p\in\left(1,p_*\right]$, is a structural fact, as suitable counterexamples in \cite{BIV2010} and \cite[Chapter~6]{DBGV-mono} show. On the other hand, we think that the non-optimality of $q_o$ given here is just a technical fact, and a different approach might provide $q_o=\frac1{p-1}$, in such a way yielding a unified description for solutions to \eqref{Eq:1:1}--\eqref{Eq:1:2}, in the full singular range $1<p<2$. Unfortunately, we could not come up with such an approach at the moment.

Finally, few words about the contents of the paper. The proof of Theorem~\ref{Thm:1:1} is given in \S~\ref{S:final}
and of Corollary~\ref{cor:holder} in \S~\ref{S:7}; in both sections we concentrate on the actual novelties, and we refer to the analogous proof in \cite{GLL}, whenever it would be a straightforward repetition of arguments already displayed elsewhere.

We devote \S~\ref{S:Prelim} to introductory materials, in particular an $L^1$ Harnack inequality; the remaining sections concern the discussion of an auxiliary problem (\S~\ref{S:Ext}), the presentation of the Harnack-type inequality which is known in the singular sub-critical range (\S~\ref{S:H-sub}), and a lower bound for a proper super-solution, proven in \cite{GLL} (\S~\ref{S:low-b}). An Appendix collects definitions, results, and some examples about the notion of capacity.
\vskip.2truecm
%%%%%%%%%%%%%%%%%%%%%%%
\begin{ack} 
{\normalfont U. Gianazza is grateful to the TIFR--CAM of Bangalore, India, where part of this work was written.
U. Gianazza was supported by the grant 2017TEXA3H\_002 ``Gradient flows, Optimal Transport and Metric Measure Structures". N. Liao was supported by the FWF--Project P31956--N32 ``Doubly nonlinear evolution equations". Both authors are grateful to the anonymous referees for their comments, which greatly helped improve the quality of the paper.}
\end{ack}
%%%%%%%%%%%%%%%%%%%%%%%%%%%%%%%%%%%%%%%%%%%%%%%%
%%%%%%%%%%%%%%%%%
\section{Preliminaries}\label{S:Prelim}
\subsection{The Definition of Solutions}\label{S:1:2}
A function
\begin{equation*}%  \label{Eq:1:4}
u\in C\big(0,T;L^2_{\loc}(E)\big)\cap L^p_{\loc}\big(0,T; W^{1,p}_{%
\loc}(E)\big)
\end{equation*}
is a local, weak sub(super)-solution to \eqref{Eq:1:1}--\eqref{Eq:1:2} if
for every compact set $K\subset E$ and every sub-interval $[t_1,t_2]\subset
(0,T]$
\begin{equation*}  %\label{Eq:1:5}
\int_K u\vp\,dx\bigg|_{t_1}^{t_2}+\int_{t_1}^{t_2}\int_K \big[-u\vp_t+\bl{A}%
(x,t,u,Du)\cdot D\vp\big]dxdt\le(\ge)0
\end{equation*}
for all non-negative test functions
\begin{equation*}
\vp\in W^{1,2}_{\loc}\big(0,T;L^2(K)\big)\cap L^p_{\loc}\big(0,T;W_o^{1,p}(K)%
\big).
\end{equation*}
This guarantees that all the above integrals %in \eqref{Eq:1:5} 
are convergent. A function $u$, which is both a local, weak sub-solution and
a local, weak super-solution, is a local, weak solution. 

For any $k\in\rr$, let
\[
(v-k)_-=\max\{-(v-k),0\},\qquad(v-k)_+=\max\{v-k,0\}.
\]
Accordingly, we notice that
\[
k-(u-k)_-=\min\{u,k\},\qquad k+(u-k)_+=\max\{u,k\}.
\]
Using \eqref{Eq:1:2}$_1$ and employing a similar method as in
({\bf A}$_6$) of \cite[Chapter II]{dibe-sv} or Lemma~1.1 of \cite[Chapter 3]{DBGV-mono},
we can show the equation \eqref{Eq:1:1} with \eqref{Eq:1:2}
is {\it parabolic}, in the sense that 
\begin{equation*}%\label{Eq:para}
	\left\{
	\begin{aligned}
		&\mbox{whenever $u$ is a local weak sub(super)-solution,}\\ 
		&\mbox{the function $k\pm(u-k)_\pm$ is a local weak sub(super)-solution, for all $\ k\in\rr$.}
		%&\mbox{$\bl{A}(x,t,u,Du)$ replaced by $\pm\bl{A}\big(x,t,k\pm(u-k)_\pm,\pm D(u-k)_\pm\big)$ and }\\
		%&\mbox{$|u|^{p-2}u$ replaced by $\pm\big|k\pm(u-k)_\pm\big|^{p-2}\big(k\pm(u-k)_\pm\big)$. }
	\end{aligned}
	\right.
\end{equation*}
%We require \eqref{Eq:1:1}--\eqref{Eq:1:2} to be {\it parabolic}: we refer to \cite{GLL}, \cite[Chapter II]{dibe-sv} or Lemma~1.1 of \cite[Chapter 3]{DBGV-mono} for more details about the notion of parabolicity, and the simple requirement, which ensures that such a notion is satisfied in our case.

%For $a,\,\theta_1,\,\theta_2 > 0$ and $(y,s)\in E_T$, we will consider 
%\begin{align*}
%\left\{
%\begin{array}{ll}
%\text{ forward  cylinders: }\ \  K_{a\rho} (y) \times (s,s+\theta_2 \rho^p],\\[3pt]
%\text { backward cylinders: }\ \  K_{a\rho} (y) \times (s-\theta_1 \rho^p,s],\\[3pt]
%\text{ centered cylinders: }\ \  K_{a\rho} (y) \times (s-\theta_1 \rho^p,s+\theta_2\rho^p].
%\end{array}
%\right.
%\end{align*}

A weak sub(super)-solution to
the Cauchy-Dirichlet problem \eqref{Eq:1:6} is a measurable function 
$$u\in C\big(0,T;L^2(E)\big)\cap 
L^p\big(0,T; W^{1,p}(E)\big)$$ satisfying
\begin{equation*}% \label{Eq:1:7}
\begin{aligned}
&\int_E u\vp(x,t)\,dx+\iint_{E_T} \big[-u\vp_t+\bl{A}%
(x,t,u,Du)\cdot D\vp\big]dxdt
\le(\ge)\int_E g\vp(x,0) \,dx
\end{aligned}
\end{equation*}
for all non-negative test functions
\begin{equation*}
\vp\in W^{1,2}\big(0,T;L^2(E)\big)\cap L^p\big(0,T;W_o^{1,p}(E)\big).
\end{equation*}
In addition, we take the boundary condition $u\le g$ ($u\ge g$) to
mean that $(u-g)_+(\cdot,t)\in W^{1,p}_o(E)$ ($(u-g)_-(\cdot,t)\in W^{1,p}_o(E)$) for
a.e. $t\in(0,T]$. A function $u$, which is both a weak sub-solution and
a weak super-solution, is a weak solution. 

We have given the definition in a global way, but all the
following arguments and results will have a local thrust: indeed, what we
are interested in, is whether solutions $u$ to
\eqref{Eq:1:6} continuously assume the given boundary
data at a single point or on some distinguished part of the lateral
boundary $S_T$ of a cylinder, but not necessarily on the whole $S_T$.
In this context,  the initial datum does not play a role.

In the sequel we will need the following comparison principle for weak (sub/super)-solutions
(see \cite[Lemma~3.1]{KiLi96}, \cite[Lemma~3.5]{KoKuPa10} and \cite[Section~2]{BBGP}). 
%\color{red}
We let $\Om$ be an open set in $\rn$, and $\Om_{S}:=\Om\times[0,S]$.
Note carefully that there is no connection between $\Om$ and the set $E$.
%%%%%%%%%%%%
\begin{lemma}[Weak Comparison Principle]\label{lem-comp}
Suppose that $v$ is a weak super-solution and $u$ is 
a weak sub-solution to \eqref{Eq:1:6} 
in $\Om_{S}$ under conditions \eqref{Eq:1:2}--\eqref{Eq:bound},
and with boundary values $g$ and $h$ respectively. If $v$ and $-u$ are lower semicontinuous on $\overline{\Om_{S}}$ and $g\ge h$ on the parabolic boundary $\partial_p \Om_{S}$, then $v\ge u$ a.e.\ in $\Om_{S}$.
\end{lemma}
\begin{remark}\label{rmk-comp}
{\normalfont As remarked in \cite{KiLi96}, the proof of the Weak Comparison Principle shows that a super-solution $v$ is greater than a sub-solution $u$ in $\Om_{S}$ if $u(x,0) \le v(x,0)$ and for each $t\in(0,S)$ the function $x\mapsto(u(x,t)-v(x,t))_+$ is in the space $W^{1,p}_o(\Om)$. See also \cite[Chapter~7]{DBGV-mono}.}
\end{remark}
%%%%%%%%%%%%
\color{black}
\subsection{Particular Super-solutions Across $S_T$}\label{S:A:1}
Fix $(x_o,t_o)\in S_T$, consider the cylinder 
\begin{equation*}%\label{Eq:cyl1}
Q=K_{32\rho}(x_o)\times[s,t_o],
\end{equation*} 
where $s$ and $t_o$ are such that $0<s<t_o\le T$, and let $\Sigma\df=S_T\cap Q$.

We extend $\bl A$ to $\widetilde{\bl A}$ defined in $Q\times\rr\times\rn$, setting
\begin{equation*}
\widetilde{\bl A}(x,t,u,\xi)=
\begin{cases}
{\bl A}(x,t,u,\xi)&\quad\text{ for a.e. }\ (x,t)\in Q\cap E_T,\,\forall\,u\in\rr,\,\forall\,\xi\in\rn,\\
|\xi|^{p-2}\xi&\quad\text{ for a.e. }\ (x,t)\in Q\backslash E_T,\,\forall\,u\in\rr,\,\forall\,\xi\in\rn.
\end{cases}
\end{equation*}
 It is apparent that 
$\widetilde{\bl A}$ satisfies conditions \eqref{Eq:1:2} and \eqref{Eq:5:1:3}--\eqref{Eq:5:1:4} in $Q\times\rr\times\rn$,
with $C_o$ and $C_1$ replaced by $\min\{1,\,C_o\}$ and $\max\{1,\,C_1\}$ respectively.

%\color{black}
We have the
following simple lemma, which has been stated in \cite[Lemma~2.1]{GLL}.
However, we present a complete proof here.

\begin{lemma}\label{Lem:2:0}
 Take any number $k$ such that
  $k\ge\sup_\Sigma g$. Let $u$ be a weak solution to the problem
  \eqref{Eq:1:6}, and define 
  \begin{displaymath}
    u_k=
    \begin{cases}
      (u-k)_+, &\text{in }Q\cap E_T, \\
      0, & \text{in }Q\setminus E_T.
    \end{cases}
  \end{displaymath}  
  Then $u_k$ is a weak sub-solution to
\eqref{Eq:1:1} in the cylinder $Q$ with $\bl A$ replaced by $\widetilde{\bl A}(x,t,k+u_k,Du_k)$. The
  same conclusion holds for the zero extension of $u_h=(h-u)_+$ for
  truncation levels $h\leq \inf_{\Sigma}g$ with $\bl A$ replaced by $-\widetilde{\bl A}(x,t,h-u_h,-Du_h)$.
\end{lemma}
\begin{proof}
This is a boundary version of the arguments in
({\bf A}$_6$) of \cite[Chapter II]{dibe-sv} or Lemma~1.1 of \cite[Chapter 3]{DBGV-mono}.
For $\ell>0$, let $\llbracket u\rrbracket_{\ell}$ denote the Steklov average of $u$
in the time variable. In terms of the Steklov averages, the weak formulation of the solution $u$ to \eqref{Eq:1:6}
can be written as
\begin{align*}
\int_{E\times\{t\}}\pl_t\llbracket u\rrbracket_{\ell}\vp\,dx+\int_{E\times\{t\}} \llbracket\widetilde{\bl A}(x,t,u, Du)\rrbracket_{\ell}\cdot D\vp\,dx=0
\end{align*}
for all $0<t<T-\ell$ and for all $\vp\in W^{1,p}_o(E)\cap L^{\infty}(E)$ (see \cite[Chapter~II, Remark~1.1]{dibe-sv} for more details on this equivalent formulation).

Let us omit the reference to $x_o$. Consider an arbitrary interval $[t_1,t_2]\subset [s,t_o]$.
Let $\z$ be a piecewise smooth function in $K_{32\rho}\times(t_1,t_2)$,
vanishing on $\pl K_{32\rho}\times(t_1,t_2)$.
It is rather easy to show that the following is an admissible test function for $\varep>0$:
\[
\{K_{32\rho}\cap E\}\times(t_1,t_2)\ni(x,t)\mapsto\vp(x,t)=
\frac{(\llbracket u\rrbracket_{\ell}-k)_+}{(\llbracket u\rrbracket_{\ell}-k)_++\varep}\z.%\in W^{1,p}_o(E)\cap L^{\infty}(E)
\]
Moreover, since $\vp(\cdot,t)$ vanishes on $K_{32\rho}\cap\pl E$ in the sense of traces
for a.e. $t\in(t_1,t_2)$ (\cite[Lemma~2.1]{GLL}),
we may extend such $\vp$ to be zero outside $E_T$, and carry over
the integral formulation into $Q$. This remark together with a time integration in $(t_1,t_2)$ gives
\begin{align*}
\int_{t_1}^{t_2}\int_{K_{32\rho}}\pl_t\llbracket u\rrbracket_{\ell}\vp\,dxdt
+\int_{t_1}^{t_2}\int_{K_{32\rho}} \llbracket{\widetilde{\bl A}}(x,t,u, Du)\rrbracket_{\ell}\cdot D\vp\,dxdt=0
\end{align*}
The first term is estimated via an integration by parts, to obtain
\begin{align*}
\int_{t_1}^{t_2}\int_{K_{32\rho}}&\pl_t\llbracket u\rrbracket_{\ell}\vp\,dxdt
=\int_{t_1}^{t_2}\int_{K_{32\rho}}\pl_t\mathfrak{h}(\llbracket u\rrbracket_{\ell},\varep)\z\,dxdt\\
&=\int_{K_{32\rho}}\mathfrak{h}(\llbracket u\rrbracket_{\ell},\varep)\z\,dx\Big|_{t_1}^{t_2}
-\int_{t_1}^{t_2}\int_{K_{32\rho}}\mathfrak{h}(\llbracket u\rrbracket_{\ell},\varep)\pl_t\z\,dxdt,
\end{align*}
where we have set
\[
\mathfrak{h}(\llbracket u\rrbracket_{\ell},\varep)=\int^{\llbracket u\rrbracket_{\ell}}_k\frac{(s-k)_+}{(s-k)_++\varep}\,ds.
\]
The second term is estimated by
\begin{align*}
\int_{t_1}^{t_2}&\int_{K_{32\rho}} \llbracket{\widetilde{\bl A}}(x,t,u, Du)\rrbracket_{\ell}\cdot D\vp\,dxdt\\
&=\int_{t_1}^{t_2}\int_{K_{32\rho}} \llbracket{\widetilde{\bl A}}(x,t,u, Du)\rrbracket_{\ell}\cdot 
\bigg[D\z\frac{(\llbracket u\rrbracket_{\ell}-k)_+}{(\llbracket u\rrbracket_{\ell}-k)_++\varep}
	+\z\frac{\varep D(\llbracket u\rrbracket_{\ell}-k)_+}{\big((\llbracket u\rrbracket_{\ell}-k)_++\varep\big)^2}\bigg]\,dxdt.
\end{align*}
Combining all the terms we arrive at
\begin{align*}
\int_{K_{32\rho}}&\mathfrak{h}(\llbracket u\rrbracket_{\ell},\varep)\z\,dx\Big|_{t_1}^{t_2}
-\int_{t_1}^{t_2}\int_{K_{32\rho}}\mathfrak{h}(\llbracket u\rrbracket_{\ell},\varep)\pl_t\z\,dxdt\\
&+\int_{t_1}^{t_2}\int_{K_{32\rho}} \llbracket{\widetilde{\bl A}}(x,t,u, Du)\rrbracket_{\ell}\cdot 
D\z\frac{(\llbracket u\rrbracket_{\ell}-k)_+}{(\llbracket u\rrbracket_{\ell}-k)_++\varep}\,dxdt\\
&=-\varep\int_{t_1}^{t_2}\int_{K_{32\rho}} \z\llbracket{\widetilde{\bl A}}(x,t,u, Du)\rrbracket_{\ell}\cdot\frac{D(\llbracket u\rrbracket_{\ell}-k)_+}{\big((\llbracket u\rrbracket_{\ell}-k)_++\varep\big)^2}\,dxdt.
\end{align*}
The conclusion is reached by first letting $\ell\to0$ and then $\varep\to0$.
One only has to notice that the right-hand side produces a non-positive quantity in this process,
due to the ellipticity of $\widetilde{\bl A}$.
\end{proof}

Let $k$ be any number such that $k\ge \sup_{\Sigma} g$, and for $u_k$ as in Lemma~\ref{Lem:2:0}, set
\begin{equation}\label{Eq:super}
\left\{
\begin{aligned}
&\mu=\sup_Q u_k,\\ 
&v:Q\rightarrow\rr_+,\quad v\df=\mu-u_k.
\end{aligned}
\right.
\end{equation}
It is not hard to verify that $v$ is a non-negative, weak super-solution to
\eqref{Eq:1:1} in $Q$. More precisely, we can write
\[
v_t-\dvg\big(-\widetilde{\bl A}(x,t,k+\mu-v,-Dv)\big)\ge0\quad\text{ weakly in }Q.
\]
\begin{remark}\upshape
As shown in the proof, any extension of $\bl A$ outside $E_T$ would work for our purpose,
provided it verifies the similar structure conditions \eqref{Eq:1:2} and \eqref{Eq:5:1:3}--\eqref{Eq:5:1:4}.
For simplicity, in the following we will denote by $\bl A$ the extended function $\widetilde{\bl{A}}$.
\end{remark}
%%%%%%%%%%%%%%
\begin{remark}
{\normalfont The choice of $k$ in the definition of $u_k$ is done in order to guarantee
that $u_k$ can be extended to zero in $Q\backslash E_T$: this yields a function which is defined on the whole $Q$, and is needed in %Proposition~\ref{Prop:A:1:1} and 
Lemma~\ref{Lm:5:1}. Therefore, any other choice of $k$ which ensures the same extension of $u$ to the whole $Q$ is allowed.}
\end{remark}
%%%%%%%%%%%%%%%%%%%%%%%%%%%%%%%%%%%%%
\subsection{An $L^1$ Harnack Inequality for Solutions when $1<p<2$}
As we have already done when stating Lemma~\ref{lem-comp}, let $\Om$ be an open set in $\rn$, and 
$\Om_{S}:=\Om\times[0,S]$. As before, that there is no connection between $\Om$ and the set $E$.
\begin{proposition}%\label{Prop:A:1:1sol} 
Suppose $u$ is a non-negative, local
weak solution to \eqref{Eq:1:1}--\eqref{Eq:1:2} in $\Omega_{S}$
for $1<p<2$.  There exists a positive constant 
$\gm$ depending only on the data $\datap$, such that 
for all $[s_1,t_1]\subset [0,S]$, and $K_{2\rho}(y)\subset\Om$
\begin{equation}\label{Eq:A:1:2sol}
\sup_{s_1<\tau<t_1}\int_{K_\rho(y)} u(x,\tau)dx\le\gm
\inf_{s_1<\tau<t_1}\int_{K_{2\rho}(y)}u(x,\tau)dx
+\gm\Big(\frac{t_1-s_1}{\rho^\lm}\Big)^{\frac1{2-p}}
\end{equation}
where 
\begin{equation*}
\lm=N(p-2)+p.
\end{equation*}
The constant $\gm=\gm(p)\to\infty$ either as $p\to2$ or as 
$p\to1$.\index{Stability of constants}
\end{proposition}
\begin{proof}
See Proposition~A.1.1 in \cite[Appendix~A]{DBGV-mono}. Here conditions~\eqref{Eq:5:1:3}--\eqref{Eq:5:1:4} are not needed in the proof.
\end{proof}
%%%%%%%%%%%%%%%%%%%%%%%%%%%%%%%%%%%%%
\subsection{An Auxiliary Problem}\label{S:Ext}
Assume $\eta$ and $\bar T$ satisfy
\[
s<\eta<\eta+\bar T\le t_o,
\]
with $s,\,t_o$ as in the definition of $Q$ in \S~\ref{S:A:1}. 
Suppose $u_o$ is a non-negative, bounded, measurable function defined in $K_{32\rho}(x_o)$, with support in $K_{2\rho}(x_o)$.
We consider the Cauchy-Dirichlet problem
\begin{equation}\label{Eq:Aux:Prob}
\left\{
\begin{aligned}
&u_\tau-\dvg\bl{A}(x,\tau+\eta,u, Du) = 0\quad \text{ weakly in }\> K_{32\rho}(x_o)\times(0,\bar T],\\
&u(\cdot,\tau)\Big|_{\partial K_{32\rho}(x_o)}=0\quad \text{ a.e. }\ \tau\in(0,\bar T],\\
&u(\cdot,0)=u_o%,\ \ \text{ where }\ \ u_o\ge0,\ \ \textcolor{red}{u_o\ \text{ bounded}},\ \ \supp u_o\subseteq K_{2\rho}(x_o).
\end{aligned}
\right.
\end{equation}
We may apply \eqref{Eq:A:1:2sol} to the above solution in the pair of cubes $K_{2\rho}(x_o)$ and $K_{4\rho}(x_o)$, 
with the choices 
\begin{equation}\label{Eq:c1}
s_1=0,\quad t_1=c_1\left[\bint_{K_{2\rho}(x_o)}u_o\,dx\right]^{2-p}\rho^p,\quad
c_1=\frac{2^{(N-1)(2-p)}}{\gm^{2-p}}.
\end{equation}
As a result, if we define $\bar T:=t_1$, then we have
\begin{equation}\label{Eq:lower}
\bint_{K_{2\rho}(x_o)} u_o\,dx\le2^{N+1}\gm
\inf_{0<\tau<\bar T}\bint_{K_{4\rho}(x_o)}u(x,\tau)\,dx.
\end{equation}
It is well-known that solutions to \eqref{Eq:Aux:Prob} extinguish in finite time.
The estimate \eqref{Eq:lower} shows solutions will not extinguish before $\bar T$.
\subsection{A Harnack-type Inequality when $1<p\le p_*$}\label{S:H-sub}
We recall here the main result of \cite{fornaro-vespri}. The supremum or infimum of a function
is meant to be the essential one. Moreover, we work in general, open, bounded sets                   
$\Om\subset\rn$ and $\Om_S:=\Om\times[0,S]$. 
At this step, there is no connection between $\Om$ and the set $E$ we have considered so far.
\begin{theorem}\label{FV}
Let $u$ be a non-negative, locally bounded, local, weak solution to \eqref{Eq:1:1}--\eqref{Eq:1:2},  \eqref{Eq:5:1:3}--\eqref{Eq:5:1:4} in $\Om_S$, with $p$ satisfying \eqref{Eq:p-range}, and let $r\ge1$ be such that
\[
\lm_r\df=N(p-2)+rp
\] 
is strictly positive. Let $(y,s)\in \Om_S$, consider $\rho>0$ such that $K_{2\rho}(y)\subset \Om$, and set
\begin{equation}\label{Eq:H:1}
\theta=c_2\left[\bint_{K_{2\rho}(y)}u(x,s)\,dx\right]^{2-p},\qquad{\sig=\left[\frac{\mathlarger\bint_{K_{2\rho}(y)} u(x,s)\,dx}{\left(\mathlarger\bint_{K_{2\rho}(y)} u^r(x,s)\,dx\right)^{\frac1r}}\right]^{\frac{pr}{\lm_r}}}.
\end{equation}
Then there exist a constant $c_2\in(0,1)$ that can be determined only in terms of the data $\datap$, and two positive constants $\gm$ and $d$, that can be determined in terms of the data $\datap$ and $r$, such that if $K_{16\rho}(y)\times[s,s+\theta\rho^p]\subset \Om_S$,
then there holds
\begin{equation*}
\inf_{K_{2\rho}(y)\times[s+\frac34\theta\rho^p,s+\theta\rho^p]} u \ge\gm\,\sig^{d}\sup_{K_{\rho}(y)\times[s+\frac12\theta\rho^p,s+\theta\rho^p]} u,
\end{equation*}
\end{theorem}
\begin{remark}\normalfont
We point out that Theorem~\ref{FV} holds not just for $p$ in \eqref{Eq:p-range}, but for any pair of $p\in(1,2)$ and $r\ge1$, such that $\lm_r>0$. See the discussion following Theorem~1.1 in \cite{fornaro-vespri}. Here we limited ourselves to the statement which is needed in our context.
\end{remark}
%%%%%%%
\begin{remark}
{\normalfont The statement given in \cite{fornaro-vespri} is slightly different, as far as the top of the cylinder, where the infimum is taken, is concerned. %\textcolor{red}{
Indeed, Theorem~1.1 of \cite{fornaro-vespri} states that
\begin{equation*}
\inf_{K_{2\rho}(y)\times[s+\frac34\theta\rho^p,s+\left(\frac34+\frac1{4^{p+1}}\right)\theta\rho^p]} u \ge\gm\,\sig^{d}\sup_{K_{\rho}(y)\times[s+\frac12\theta\rho^p,s+\theta\rho^p]} u.
\end{equation*}
However, a careful inspection of the proof given in \cite{fornaro-vespri}, and a comparison with the proof of the analogous result given in \cite[Chapter~6, Section~11]{DBGV-mono}, shows that  the top of the cylinders can be taken to be the same, both for the supremum and the infimum. We refrain from giving further details here.}%}
\end{remark}
%%%%%%%%%%%%%%%%%%%%%%%%%%%%%%%%%%%%%
\subsection{A Lower Bound on the Super-Solution $v$ Defined in \eqref{Eq:super}}\label{S:low-b}
From here on we let 
\[
c\df=\min\{c_1,\,c_2\},
\] 
which are given in \eqref{Eq:c1} and in \eqref{Eq:H:1}, and $\bar T$ is accordingly redefined through $c$. We point out that either it remains the same, or it is shortened. 
The following lemma has been shown in \cite[Lemma 5.1]{GLL}. 
\begin{lemma}\label{Lm:5:1}
Let $Q$, $u_k$, $\mu$, $v$ be as in \S~\ref{S:A:1}, consider $(x_o,\eta)\in\Sigma$ with $s<\eta\le t_o$, let
\[
\theta=c\left[\bint_{K_{2\rho}(x_o)}v(x,\eta)\,dx\right]^{2-p},
\]
and assume that 
\begin{equation*}
s\le \eta-\theta\rho^p<\eta\le t_o. 
\end{equation*}
Then, if we let 
\begin{align*}
&\dl(\rho)\df=\frac{{\rm cap}_p\big(K_{\rho}(x_o)\backslash E,K_{\frac32\rho}(x_o)\big)}{{\rm cap}_p\big(K_{\rho}(x_o),K_{\frac32\rho}(x_o)\big)},
\end{align*}
there exists a constant $\gm>1$ that depends only on the data $\datap$, such that
\begin{equation}\label{Eq:low-bd1}
\mu\,[\dl(\rho)]^{\frac1{p-1}}\le\gm\, \bint_{K_{2\rho}(x_o)}v(x,\eta)\,dx.
\end{equation}
\end{lemma}
\begin{remark}\normalfont
Lemma~\ref{Lm:5:1} holds for all $1<p<2$, and does not depend on conditions~\eqref{Eq:5:1:3}--\eqref{Eq:5:1:4}.
\end{remark}
%%%%%%%%%%%%%%%%%%%%%%%%%%%%%%%%%%%%%
\section{Proof of Theorem~\ref{Thm:1:1}}\label{S:final}
Consider \eqref{Eq:Aux:Prob} with $u_o(x)=v(x,\eta)\chi_{K_{2\rho}(x_o)}(x)$. 
%In light of \eqref{Eq:lower},  it suffices to work in the shorter cylinder $K_{32\rho}(x_o)\times[0,T_{\text{low}}]$ instead of $K_{32\rho}(x_o)\times[0,\bar T]$, and 
Assume that 
\begin{align*}
s<\eta<\eta+\bar T\le t_o,\quad\text{ where }\bar T=c\left[\bint_{K_{2\rho}(x_o)}v(x,\eta)\,dx\right]^{2-p}\rho^p.
\end{align*}
%This ensures that $K_{32\rho}(x_o)\times(t_o,t_o+T_{\text{low}}]\subseteq Q$, with $Q$ defined in \S~\ref{S:A:1}.
%\noi If we rely on \eqref{Eq:A:1:2sol} over $K_{2\rho}(x_o)$ and $K_{4\rho}(x_o)$, with \textcolor{red}{$s_1=0$ and $t_1=\frac12 T_{\text{low}}$}, then for any $\tau\in[0,\frac12T_{\text{low}}]$, due to the explicit expression of $c_1$ in \eqref{Eq:lower}, we have
According to \eqref{Eq:lower}, we have
\begin{align*}
\frac1{2^{N+1}\gm}\bint_{K_{2\rho}(x_o)} u_o(x)\,dx\le\inf_{0<\tau<\bar T}\bint_{K_{4\rho}(x_o)}u(x,\tau)\,dx.
\end{align*}
%If we let $\tau\in[\frac14T_{\text{low}},\frac12T_{\text{low}}]$, 
This together with \eqref{Eq:low-bd1} yields
\begin{align*}
\sup_{K_{4\rho}(x_o)\times[\frac12\bar T,\bar T]} u&\ge\frac1{2^{N+1}\gm}\,\bint_{K_{2\rho}(x_o)} u_o(x)\,dx\\
&=\frac1{2^{N+1}\gm}\,\bint_{K_{2\rho}(x_o)}v(x,\eta)\,dx\\
&\ge\gm\mu[\dl(\rho)]^{\frac1{p-1}},
\end{align*}
where $\gm$ in the last line takes into account all the constants. 

We can now apply Theorem~\ref{FV} with $\Om=K_{32\rho}(x_o)$ and $(y,s)=(x_o,0)$, to conclude that
\begin{equation*}
\inf_{K_{2\rho}(x_o)\times[\frac34\bar T,\bar T]} u \ge\gm\sig^{d}\mu[\dl(\rho)]^{\frac1{p-1}},
\end{equation*}
where
\[
\sig\df=\left[\frac{\mathlarger\bint_{K_{2\rho}(x_o)} u_o\,dx}{\left(\mathlarger\bint_{K_{2\rho}(x_o)} u^r_o\,dx\right)^{\frac1r}}\right]^{\frac{pr}{\lm_r}},
\]
$r>1$ is such that $\lm_r>0$, and $d$ depends on the data and on $r$. 
Since 
\[
\bint_{K_{2\rho}(x_o)} u_o\,dx\overset{\eqref{Eq:low-bd1}}{\ge}\frac1\gm\mu[\dl(\rho)]^{\frac1{p-1}},\qquad\left(\bint_{K_{2\rho}(x_o)} u^r_o\,dx\right)^{\frac1r}\le
\mu^{\frac{r-1}r}\left(\bint_{K_{2\rho}(x_o)}u_o\,dx\right)^{\frac1r},
\]
we eventually obtain
\begin{equation*}
\inf_{K_{2\rho}(x_o)\times[\frac34\bar T,\bar T]} u \ge\gm\mu[\dl(\rho)]^{\frac1{p-1}\left(1+d\frac{p(r-1)}{\lm_r}\right)}.
\end{equation*}
By Lemma~\ref{lem-comp} and Remark~\ref{rmk-comp}, we have $v(x,\tau+\eta)\ge u(x,\tau)$ a.e. in $K_{32\rho}(x_o)\times[0,\bar T]$. Therefore, setting 
\begin{equation}\label{Eq:defin:qo}
q_o\df=\frac1{p-1}\left(1+d\frac{p(r-1)}{\lm_r}\right),
\end{equation} 
we conclude
\begin{equation}\label{Eq:final-1}
\mu[\dl(\rho)]^{q_o}\le\gm\inf_{K_{2\rho}(x_o)}v(\cdot,\tau)
\end{equation}
for all 
\begin{equation}\label{Eq:final-2}
%\begin{aligned}
\tau\in\left[\eta+\frac34\theta\rho^p,\eta+\theta\rho^p\right],\quad\text{ with }\quad\theta=c\left[\bint_{K_{2\rho}(x_o)}v(x,\eta)\,dx\right]^{\scriptstyle2-p}.
%\end{aligned}
\end{equation}
%where $c\in(0,1)$ is the constant of Theorem~\ref{FV}. 

Since \eqref{Eq:final-1}--\eqref{Eq:final-2} correspond to (5.2)--(5.3) of \cite{GLL}, from here on we can conclude the proof as in \cite[Section~6]{GLL}, up to minor adaptations of the main argument, which we now discuss.

A key step of the induction argument in \cite[Section~6.3]{GLL} is to fit a sequence of
intrinsically scaled cylinders in one another. There we chose to present this trick in a sequence
of {\it centered} cylinders, such that the oscillation of $u$ is quantitatively reduced therein.
Here we take the opportunity to point out that actually we may choose a nested sequence of {\it backward} cylinders 
which will do the same job.
The modification is based on the following lemma.
\begin{lemma}\label{Lm:7:1}
Fix $(x_o,t_o)\in S_T$, $s$ such that $0<s<t_o\le T$, let $Q$, $u_k$, $\mu$, $v$ be as in \S~\ref{S:A:1}, consider %$\pto\in\Sigma$ with $s<t_o$, let
$\widetilde\theta=c\mu^{2-p}$ with $c$ as in \S~\ref{S:low-b},
%where $c$ is the same constant as in \eqref{Eq:H:1}, 
and assume that 
\begin{equation*}%\label{Eq:5:0}
s\le t_o-2\widetilde\theta(2\rho)^p<t_o. 
\end{equation*}
Then, there exists a constant $\widetilde\gm>1$ that depends only on the data $\datap$, such that
\begin{equation*}%\label{Eq:low-bd2}
\mu\,[\dl(\rho)]^{q_o}\le\widetilde\gm \inf_{K_{2\rho}(x_o)} v(\cdot,\tau)
\quad\text{ for all }\ \tau\in[t_o-2\widetilde\theta\rho^p,t_o].
\end{equation*}
\end{lemma}
\begin{proof}
For simplicity, let us assume \eqref{Eq:final-1} holds for $\tau=\eta+\theta\rho^p$, where
$\theta$ is defined in \eqref{Eq:final-2}.
Consider the closed and bounded interval
\[
I\df=[t_o-3\widetilde\theta\rho^p,t_o]\subset[t_o-2\widetilde\theta(2\rho)^p,t_o].
\]
Using $0\le v\le \mu$, it is not hard to see that, for any $\eta\ge t_o-3\widetilde\theta\rho^p$, we have
\[
\eta-c\bigg[\bint_{K_{2\rho}(x_o)}v(x,\eta)\,dx\bigg]^{2-p}\rho^p\ge t_o-2\widetilde\theta(2\rho)^p.
\]
Consequently, Lemma~\ref{Lm:5:1} can be applied for all $\eta\in I$.

Next, we introduce the function $f:I\to\rr$ defined by
\[
f(\eta)=\eta+c\bigg[\bint_{K_{2\rho}(x_o)}v(x,\eta)\,dx\bigg]^{2-p}\rho^p.
\]
It is straightforward to check that $f\in C(I)$.
%Attention is called for a slight abuse of the symbol $t$, lest confusion arises.
%Indeed, it appeared in the definition of $Q$ in \S~\ref{S:A:1}, which has been fixed to be $t_o$ now.
%On the other hand, we use $t$ once more as a running variable in the proof.
Using $0\le v\le \mu$ again, a simple calculation yields that
\begin{equation*}%\label{Eq:6:4}
\begin{aligned}
&\min_I f\le f(t_o-3\widetilde\theta\rho^p)\le t_o-2\widetilde\theta\rho^p,\\
&\max_I f\ge f(t_o)\ge t_o.
%&\max_I f\le t_o+2\theta\rho^2.
\end{aligned}
\end{equation*}
Thus, by the mean value theorem, there exists $t_*\in I$ such that $f(t_*)=t_o$. Without loss of generality,
we may assume $t_*$ is the smallest among such numbers in $I$. Accordingly, we define
$I_*:=[t_o-3\widetilde\theta\rho^p,t_*]$ and $J:=[t_o-2\widetilde\theta\rho^p,t_o]$.
%By Lemma~\ref{Lm:5:1},
%\[
%\mu\,[\dl(\rho)]^{\frac1{p-1}}\le\gm\,\dashint_{K_{2\rho}(x_o)}v(\cdot,\tau)\,dx\quad\text{ for all }
%\tau\in I_*\df{=}[t_o-3\widetilde\theta\rho^p,t_*]
%\]
As a result,  there holds $J\subset f(I_*)$, and we can conclude that $f(\cdot)$ 
attains all values of $J$ as its argument ranges over $I_*$. 
Moreover, since $t_*$ is the first time when $f(t_*)=t_o$ holds, we must have $f(\eta)<t_o$
for all $\eta\in I_*\setminus\{t_*\}=[t_o-3\widetilde\theta\rho^p,t_*)$. 
Consequently, we are allowed to apply \eqref{Eq:final-1}--\eqref{Eq:final-2} for all $\eta\in I_*$,
and conclude that
\[
\mu\,[\dl(\rho)]^{q_o}\le\widetilde\gm\inf_{K_{2\rho}}v(\cdot,\tau)\quad\text{ for all }\tau\in[t_o-2\widetilde\theta\rho^p,t_o],
\]
for some properly defined positive constant $\widetilde\gm$.
\end{proof}
\begin{remark}
{\normalfont Due to the estimates we relied upon, it is apparent that we have a sharp disconnect in the characterization of the boundary behavior between $p\in\left(1,p_*\right]$ considered here, and $p>p_*$ studied in \cite{GLL, GL}.
}
\end{remark}
\begin{remark}
{\normalfont If \eqref{Eq:1:1} reduces to the prototype parabolic $p$-laplacian \eqref{p-lapl},
%\[
%u_t-\dvg(|Du|^{p-2}Du)=0,
%\]
then by (2.10) of \cite{BIV2010}, $d$ has an explicit expression, namely $d=1+\frac{\lm_r}{pr(2-p)}$, and consequently also $q_o$ in \eqref{Eq:defin:qo}.
}
\end{remark}
%%%%%%%%%%%%%%%%%%%%%%%%%%%%%%%%%%%%%
\section{Proof of Corollary~\ref{cor:holder}}\label{S:7}
We first present the main ingredient of the proof in the following Lemma~\ref{Lm:8:1}, which plays a similar
role in the induction argument as Lemma~\ref{Lm:7:1}.
Meanwhile, we emphasize that Lemma~\ref{Lm:8:1} holds for all $1<p<2$ without 
conditions~\eqref{Eq:5:1:3}--\eqref{Eq:5:1:4}.
Instead of Harnack's inequality, we employ the {\it expansion of positivity} for non-negative super-solutions
to \eqref{Eq:1:1}--\eqref{Eq:1:2};
see \cite[Chapter~4, Proposition~5.1]{DBGV-mono}.
\begin{lemma}\label{Lm:8:1}
Let $Q$, $u_k$, $\mu$, $v$ be as in \S~\ref{S:A:1}, consider $\pto\in\Sigma$ with $s<t_o$, let
$\widetilde\theta=c\mu^{2-p}$ with $c$ as in \S~\ref{S:low-b},
%where $c$ is the same constant as in \eqref{Eq:H:1}, 
assume that 
\begin{equation*}%\label{Eq:5:0}
s\le t_o-2\widetilde\theta(2\rho)^p<t_o,
\end{equation*}
and that the complement of $E$ is uniformly $p$-fat.
Then, there exists a constant $\widetilde\gm>1$ that depends only on the data $\datap$
and $\gm_o$, such that for all $0<\rho<\bar\rho$,
\begin{equation*}%\label{Eq:low-bd2}
\mu\le\widetilde\gm \inf_{K_{2\rho}(x_o)} v(\cdot,\tau)
\quad\text{ for all }\ \tau\in[t_o-2\widetilde\theta\rho^p,t_o].
\end{equation*}
\end{lemma}
\begin{proof}
As in the proof of Lemma~\ref{Lm:7:1}, we define the interval $I$,
 the function $f\in C(I)$, and verify that Lemma~\ref{Lm:5:1}
 can be applied, i.e.
 \[
 \mu[\dl(\rho)]^{\frac1{p-1}}\le\gm\,\bint_{K_{2\rho}(x_o)}v(x,\tau)\,dx\quad\text{ for all }\tau\in I=[t_o-3\widetilde\theta\rho^p,t_o].
 \] 
 Our proof departs from here.
Let $\sig\in(0,1)$ to be fixed later.
We use the uniform $p$-fatness of the complement of $E$ 
to estimate for all $\tau\in I$ and $0<\rho<\bar\rho$,
\begin{align*}
\mu\gm_o^{\frac1{p-1}}&\le\gm\,\bint_{K_{2\rho}(x_o)}v(x,\tau)\,dx\\
&=\frac{\gm}{|K_{2\rho}|}\int_{K_{2\rho}(x_o)}v(x,\tau)\chi_{[v>\sig\mu]}\,dx+
\frac{\gm}{|K_{2\rho}|}\int_{K_{2\rho}(x_o)}v(x,\tau)\chi_{[v\le\sig\mu]}\,dx\\
&\le\gm\mu\frac{|[v(\cdot,\tau)>\sig\mu]\cap K_{2\rho}(x_o)|}{|K_{2\rho}|}+\gm\sig\mu.
\end{align*}
Now we choose $\gm\sig=\frac12\gm_o^{\frac1{p-1}}$, and after simple manipulation we arrive at
\[
|[v(\cdot,\tau)>\sig\mu]\cap K_{2\rho}(x_o)|\ge\sig|K_{2\rho}|\quad\text{ for all }\tau\in I.
\]
Since $v$ is a non-negative super-solution to \eqref{Eq:1:1}--\eqref{Eq:1:2}
in $Q$, an application of \cite[Chapter~4, Proposition~5.1]{DBGV-mono} yields the desired conclusion.
\end{proof}

Based on Lemma~\ref{Lm:8:1}, the proof of Corollary~\ref{cor:holder}
can be completed by an induction argument, just like in \cite{GLL}: we refrain from further elaboration, to avoid repetition.
%%%%%%%%%%%%%%%%%%%%%%%%%%%%%%%%%%%%%
\appendix
\appendixpage
\section{Capacity and Potential Theory}\label{S:cap}
The concept of capacity plays a key role in potential theory and particularly,
in understanding the local behavior of Sobolev functions. In a sense, it takes the place
of measure in Egorov and Lusin type theorems for Sobolev functions.
Decisive here for the boundary behavior of solutions to the parabolic $p$-laplacian type equations is the $p$-capacity defined by
\[
{\rm cap}_p(K, \Om)=\inf_{u\in W}\int_\Om |Du|^p\,dx
\]
where $K$ is a compact subset of the open set $\Om$ in $\rn$
and $W=\{u\in C_o^{\infty}(\Om):u\ge1\text{ on }K\}$.
After proper approximations, the above minimization could take place over $W^{1,p}_o(\Om)$ instead of $C_o^{\infty}(\Om)$
and the capacity of an arbitrary subset of $\Om$
can be formulated based on that of compact subsets, cf. \cite[Chapter~2]{HKM} and \cite[Chapter~2]{MZ}.

%\textcolor{blue}{
%When $q_o=\frac1{p-1}$ and one considers, the notion of a Wiener point $x_o\in\partial E$ is more commonly known in Potential Theory as {\it $p$-thickness of $E$ at $x_o$} for $p\in(1,N]$ 
For a domain $E\subset\rn$ and a point $x\in \rn$, in potential theory (cf.~\cite[Definition~6.3.7]{AH} and \cite[Definition~2.47]{MZ}) 
we say $E$ is $p$-thick at $x$ if
\[
\int_0^1\bigg[\frac{{\rm cap}_p\big(K_{\rho}(x)\cap E,K_{\frac32\rho}(x)\big)}{{\rm
    cap}_p\big(K_{\rho}(x),K_{\frac32\rho}(x)\big)}\bigg]^{\frac1{p-1}}\,\frac{d\rho}{\rho}=\infty.
\]
Even though the introduction of this concept was motivated by Wiener's celebrated work \cite{W} about the boundary regularity for the laplacian, nevertheless it has witnessed a development largely independent of the theory of partial differential equations.
%the $p$-fine topology associated with the Sobolev space $W^{1,p}$ (for details, see \cite[Chapter~2]{MZ} or \cite[Section~2]{frehse}). 
It is vacuous that the domain $E$ is $p$-thick at its interior point, while not so for points in $\overline{E}^c$.
Hence the most interesting part concerning thickness lies in the boundary $\pl E$. 
A deep result in potential theory--the Kellogg property--asserts that the domain $E$ is $p$-thick at every point of $\pl E$
except a set of zero $p$-capacity, cf.~\cite[\S~6.3]{AH}.
%This, together with Theorem~\ref{Thm:1:1}, implies that weak solutions to \eqref{Eq:1:6}
%are continuous at every $(x,\cdot)\in S_T$ after excluding a set in $\pl E$ of zero $(1+q_o^{-1})$-capacity.

The peculiar feature of the notion of thickness is that it is purely geometric; unfortunately, apart from easy situations, it is far from trivial to provide explicit examples of $p$-thickness. Perhaps, one of the most interesting cases is given by the so-called {\it outer corkscrew condition} (OCC): we say that $x_o\in\partial E$ satisfies (OCC), if there exist $M\ge2$ and $r_o>0$, such that for any $r\in(0,r_o)$ there exists $a_r\in E^c$ satisfying $\frac rM<|a_r-x_o|<r$ and $\dist(a_r,\partial E)>\frac rM$. Then, by \cite[Theorem~6.31]{HKM} it is not hard to prove that $x_o$ is a Wiener point. A domain $E$ that has an outward spike satisfies an outer corkscrew condition at the tip of the spike. Other examples of Wiener points are discussed in \cite[Corollary~11.25]{BB}.

\bye

%% file: dibe_1009.tex
\newtheorem{proposition}{Proposition}[section]
\newtheorem{theorem}{Theorem}[section]
\newtheorem{lemma}{Lemma}[section]
\newtheorem{corollary}{Corollary}[section]
\newtheorem{remark}{Remark}[section]
\numberwithin{equation}{section}
\numberwithin{theorem}{section}
\numberwithin{proposition}{section}
\numberwithin{lemma}{section}
\numberwithin{remark}{section}
\newcommand{\dsty}{\displaystyle}

\newcommand{\al}{\alpha}

\newcommand{\gm}{\gamma}
\newcommand{\dl}{\delta}

\newcommand{\lm}{\lambda}
\newcommand{\Lm}{\Lambda}

\newcommand{\varep}{\varepsilon}

\newcommand{\vp}{\varphi}
\newcommand{\sig}{\sigma}

\newcommand{\om}{\omega}
\newcommand{\Om}{\Omega}

\newcommand{\z}{\zeta}

\newcommand{\df}[1]{\buildrel\mbox{\small def}\over{#1}}

\newcommand{\rr}{\mathbb{R}}
\newcommand{\rn}{\rr^N}

\newcommand{\bl}[1]{\mathbf{#1}}
\newcommand{\dvg}{\operatorname{div}}

\newcommand{\essup}{\operatornamewithlimits{ess\,sup}}
\newcommand{\essinf}{\operatornamewithlimits{ess\,inf}}
\newcommand{\essosc}{\operatornamewithlimits{ess\,osc}}
\newcommand{\osc}{\operatornamewithlimits{osc}}

\newcommand{\loc}{\operatorname{loc}}

\newcommand{\dist}{\operatorname{dist}}

\newcommand{\pl}{\partial}

\newcommand{\intl}{\int\limits}

\def\Xint#1{\mathchoice
    {\XXint\displaystyle\textstyle{#1}}%
    {\XXint\textstyle\scriptstyle{#1}}%
    {\XXint\scriptstyle\scriptscriptstyle{#1}}%
    {\XXint\scriptscriptstyle\scriptscriptstyle{#1}}%
    \!\int}
\def\XXint#1#2#3{\setbox0=\hbox{$#1{#2#3}{\int}$}
    \vcenter{\hbox{$#2#3$}}\kern-0.5\wd0}
\def\bint{\Xint-}
\def\dashint{\Xint{\raise4pt\hbox to7pt{\hrulefill}}}
\def\dashiint{\bint\kern-0.15cm\bint}
\newcommand{\ovl}[3]{\int_{#1}^{#2}\kern-#3pt\raise4pt\hbox to7pt{\hrulefill}\ }

\newcommand{\ovll}[3]{\intl_{#1}^{#2}\kern-#3pt\raise4pt\hbox to7pt{\hrulefill}\ }

\newcommand{\tvl}[2]{\iint_{#1}\kern-#2pt\raise4pt\hbox to7pt{\hrulefill}\ }

\newcommand{\bye}{\end{document}}

%% file: harnack_mono.tex
\newcommand{\datap}{\{p,N,C_o,C_1\}}

\newcommand{\pto}{(x_o,t_o)}
\newcommand{\tvls}[2]{\iint_{#1}\kern-#2pt\raise4pt\hbox to15pt{\hrulefill}\ }

\newcommand{\ibtr}{\intl_0^t\!\intl_{K_{2\rho}}\kern-4pt}